\pdfoutput=1
\documentclass[12pt,letter]{amsart}
\usepackage{xcolor}
\usepackage[all,color]{xy}
\usepackage{amsmath,amssymb,amsthm,amscd,amsxtra,amsfonts,mathrsfs,graphicx,enumerate,bm,slashed,array,setspace,amsfonts} 
\input xy
\xyoption{all}
\newtheorem{Theorem}{Theorem}[section]
\newtheorem{Lemma}[Theorem]{Lemma}
\newtheorem{Proposition}[Theorem]{Proposition}
\newtheorem{Corollary}[Theorem]{Corollary}
\newtheorem{Example}[Theorem]{Example}
\newtheorem{Remark}[Theorem]{Remark}
\newtheorem{Definition}[Theorem]{Definition}

\newtheorem{Notation}[Theorem]{Notation}
\newtheorem{Construction}[Theorem]{Construction}
\newtheorem*{Theorem A}{Theorem A}

\newcommand*{\overbar}[1]{\mkern 1.5mu\overline{\mkern-1.5mu#1\mkern-1.5mu}\mkern 1.5mu}
\advance\evensidemargin-.5in
\advance\oddsidemargin-.5in
\advance\textwidth1in

\begin{document}

\author{Charlie Beil}
\address{Institut f\"ur Mathematik und Wissenschaftliches Rechnen, Universit\"at Graz, Heinrichstrasse 36, 8010 Graz, Austria.}
 \email{charles.beil@uni-graz.at}
 \title[Cyclic contractions of dimer algebras always exist]{Cyclic contractions of dimer algebras\\ always exist}
 \date{}

\begin{abstract}
We show that every nondegenerate dimer algebra $A$ on a torus admits a cyclic contraction to a cancellative dimer algebra.
This implies, for example, that $A$ is Calabi-Yau if and only if it is noetherian; and that the center of $A$ has Krull dimension $3$.
\end{abstract}

\maketitle

\section{Introduction}

The main objective of this article is to show that every nondegenerate dimer algebra on a torus admits a cyclic contraction to a cancellative (i.e., consistent) dimer algebra.
Dimer algebras were introduced in string theory \cite{HK,FHMSVW,FHVWK}, and have found wide application to many areas of mathematics, such as noncommutative resolutions \cite{B4,B6,Bo2,Br,IN}, the McKay correspondence \cite{CBQ,IU}, cluster algebras and categories \cite{BKM,GK}, number theory \cite{BGH}, and mirror symmetry \cite{Bo,FHKV,FU}.

A dimer algebra $A = kQ/I$ is a quiver algebra whose quiver $Q$ embeds into a compact surface, with relations $I$ defined by a potential (see Definition \ref{dimer def}); in this article we will assume that the surface is a torus.
A dimer algebra is said to be nondegenerate if each arrow is contained in a perfect matching.

Let $A = kQ/I$ and $A' = kQ'/I'$ be nondegenerate dimer algebras, and suppose $Q'$ is obtained from $Q$ by contracting a set of arrows $Q_1^* \subset Q_1$ to vertices.
This operation defines a $k$-linear map of path algebras
$$\psi: kQ \to kQ'.$$
If $\psi(I) \subseteq I'$, then $\psi$ induces a $k$-linear map of dimer algebras, called a \textit{contraction},
$$\psi: A \to A'.$$
If, in addition, $A'$ is cancellative and $\psi$ preserves the so-called cycle algebra, then $\psi$ is called a \textit{cyclic contraction}.
An example of a cyclic contraction is given in Figure \ref{example}.
Cyclic contractions were introduced in \cite{B1}, and have been an essential tool in the study of non-cancellative dimer algebras.
Our main result is the following.

\begin{Theorem} \label{main}
Every nondegenerate dimer algebra on a torus admits a cyclic contraction.
\end{Theorem}

\begin{figure}
$$\begin{array}{ccc}
\xy 0;/r.365pc/:
(-12,6)*+{\text{\scriptsize{$2$}}}="1";(0,6)*+{\text{\scriptsize{$1$}}}="2";(12,6)*+{\text{\scriptsize{$2$}}}="3";
(-12,-6)*+{\text{\scriptsize{$1$}}}="4";(0,-6)*+{\text{\scriptsize{$2$}}}="5";(12,-6)*+{\text{\scriptsize{$1$}}}="6";
(-6,0)*{\cdot}="7";
{\ar_{a}"1";"4"};{\ar^{}"4";"7"};{\ar_{c}"7";"1"};{\ar^{b}"1";"2"};{\ar^{}"2";"7"};{\ar_{b}"5";"6"};{\ar"6";"2"};{\ar^{a}"3";"6"};
{\ar^{\delta}@[green]"7";"5"};{\ar@{..>}"7";"5"};
\endxy
 & \stackrel{\psi}{\longrightarrow} &
\xy 0;/r.365pc/:
(-12,6)*+{\text{\scriptsize{$2$}}}="1";(0,6)*+{\text{\scriptsize{$1$}}}="2";(12,6)*+{\text{\scriptsize{$2$}}}="3";
(-12,-6)*+{\text{\scriptsize{$1$}}}="4";(0,-6)*+{\text{\scriptsize{$2$}}}="5";(12,-6)*+{\text{\scriptsize{$1$}}}="6";
{\ar^{}"1";"4"};{\ar^{}"4";"5"};{\ar^{}"5";"1"};{\ar^{}"1";"2"};{\ar^{}"2";"5"};{\ar^{}"5";"6"};{\ar^{}"6";"2"};{\ar^{}"2";"3"};{\ar^{}"3";"6"};
\endxy
\\
Q & \ \ \ \ & Q' \\
\end{array}$$
\caption{(Example \ref{first ex}.)
A cyclic contraction $\psi: A \to A'$. 
Both quivers are drawn on a torus, and $Q'$ is obtained from $Q$ by contracting the arrow $\delta \in Q_1$.
$\delta$ is the only nonrigid arrow in $Q$.}
\label{example}
\end{figure}

This theorem is important because it implies that the results of \cite{B1,B2,B3,B5,B7}, which assume the existence of cyclic contractions, hold for every nondegenerate dimer algebra.
For example, suppose $A$ admits a cyclic contraction $\psi: A \to A'$; then $A$ is cancellative if and only if $A$ is noetherian, if and only if its center $Z$ is noetherian, if and only if $A$ is a finitely generated $Z$-module \cite[Theorem 1.1]{B3}.
Furthermore, if $A$ is non-cancellative, then $Z$ is nonnoetherian of Krull dimension 3, generically Gorenstein, and contains precisely one closed point $\mathfrak{z}_0$ of positive geometric dimension \cite[Theorem 1.1]{B5}.
In this case, $A$ is locally Morita equivalent to $A'$ away from $\mathfrak{z}_0$, and the Azumaya locus of $A$ coincides with the intersection of the Azumaya locus of $A'$ and the noetherian locus of $Z$ \cite[Theorem 1.1]{B4}.

We emphasize two points regarding the structure of the map $\psi: A \to A'$ and the cycle algebra $S$, assuming $\psi$ is nontrivial.

$\bullet$ Consider the idempotent
$$\epsilon := 1_A - \sum_{\delta \in Q_1^*}e_{\operatorname{h}(\delta)}.$$
Although $\psi$ itself is only a $k$-linear map and not an algebra homomorphism, the restriction to the subalgebra
$$\psi: \epsilon A \epsilon \to A'$$
is an algebra homomorphism.
This restriction becomes an algebra isomorphism under localizations away from $\mathfrak{z}_0$ \cite[Proposition 2.12.1]{B4}.

$\bullet$ The cycle algebra $S$ is isomorphic to the center of $A'$, and is a depiction of the reduced center $Z/\operatorname{nil}Z$ of $A$ \cite[Theorem 1.1]{B5}.
Let
$$\mathbb{S}(A) \subset \operatorname{Rep}_{1^{Q_0}}(A) \ \ \ \text{ and } \ \ \ \mathbb{S}(A') \subset \operatorname{Rep}_{1^{Q'_0}}(A')$$
be the open subvarieties consisting of simple modules over $A$ and $A'$, respectively.
Denote by $\overbar{\mathbb{S}(A)}$ and $\overbar{\mathbb{S}(A')}$ their Zariski closures.
Then $S$ is isomorphic to the $\operatorname{GL}$-invariant rings \cite[Theorem 3.14]{B4}
\begin{equation} \label{S cong}
S \cong k[ \overbar{\mathbb{S}(A)} ]^{\operatorname{GL}} \cong k[ \overbar{\mathbb{S}(A')} ]^{\operatorname{GL}}.
\end{equation}

\begin{Remark} \rm{
In the context of a four-dimensional $\mathcal{N} = 1$ abelian quiver gauge theory with quiver $Q$, the mesonic chiral ring is a commutative algebra generated by all the cycles in $Q$ modulo the superpotential relations $I$.
Theorem \ref{main} then states, loosely, that every low energy non-superconformal dimer theory can be Higgsed to a superconformal dimer theory with the same mesonic chiral ring.
(The mesonic chiral ring is not quite the same as the cycle algebra, however; see \cite[Remark 3.15]{B4}.)
}\end{Remark}

We give a brief outline of our proof of Theorem \ref{main}.
To obtain a cyclic contraction of a dimer algebra $A = kQ/I$, we form a sequence of contractions starting with $Q$, where in each iteration a single arrow is contracted; an example is given in Figure \ref{example2}.
Each contracted arrow $\delta$ has the property that each perfect matching $D$ containing $\delta$ can be `moved off' of $\delta$ and onto a different set of arrows $D'$, in such a way that $D$ and $D'$ are identical perfect matchings from the perspective of the cycles in the quiver.
Since $D$ can be transformed into $D'$ in this way, we call $D$ a `nonrigid perfect matching'.
Furthermore, we say $\delta$ is a `nonrigid arrow', since we may contract $\delta$ without changing the underlying cycle structure of $A$.

We then show that a perfect matching $D$ is nonrigid if and only if it is simple, that is, $Q \setminus D$ supports a simple representation of dimension vector $(1,\ldots,1)$.
Moreover, it was shown in \cite{B3} that if each arrow of a dimer algebra is contained in a simple matching, then it is cancellative (in fact, this is a necessary and sufficient condition for cancellativity).
Therefore, by running the sequence of single arrow contractions until there are only rigid arrows remaining--in particular, until each arrow is contained in a simple matching--we end with a dimer algebra that is cancellative and has the same cycle structure as $A$.

\begin{figure}
$$\begin{array}{ccccc}
\xy 0;/r.365pc/:
(-12,6)*+{\text{\scriptsize{$2$}}}="1";(0,6)*+{\text{\scriptsize{$1$}}}="2";(12,6)*+{\text{\scriptsize{$2$}}}="3";
(-12,-6)*+{\text{\scriptsize{$1$}}}="4";(0,-6)*+{\text{\scriptsize{$2$}}}="5";(12,-6)*+{\text{\scriptsize{$1$}}}="6";
(0,2)*{\cdot}="7";(0,-2)*{\cdot}="8";
{\ar"4";"1"};{\ar"1";"2"};{\ar"3";"2"};{\ar"6";"3"};{\ar"5";"6"};{\ar"5";"4"};
{\ar"2";"7"};{\ar@[blue]"5";"8"};{\ar@[green]"8";"7"};{\ar@{..>}"8";"7"};{\ar@[blue]@/^1.5pc/"7";"5"};{\ar@[green]@/_1.5pc/"7";"5"};{\ar@{..>}@/_1.5pc/"7";"5"};
\endxy
 & \stackrel{\psi_0}{\longrightarrow} &
\xy 0;/r.365pc/:
(-12,6)*+{\text{\scriptsize{$2$}}}="1";(0,6)*+{\text{\scriptsize{$1$}}}="2";(12,6)*+{\text{\scriptsize{$2$}}}="3";
(-12,-6)*+{\text{\scriptsize{$1$}}}="4";(0,-6)*+{\text{\scriptsize{$2$}}}="5";(12,-6)*+{\text{\scriptsize{$1$}}}="6";
(0,2)*{\cdot}="7";
{\ar"4";"1"};{\ar"1";"2"};{\ar"3";"2"};{\ar"6";"3"};{\ar"5";"6"};{\ar"5";"4"};
{\ar"2";"7"};{\ar@[blue]"5";"7"};{\ar@[blue]@/^1.5pc/"7";"5"};{\ar@[green]@/_1.5pc/"7";"5"};{\ar@{..>}@/_1.5pc/"7";"5"};
\endxy
 & \stackrel{\psi_1}{\longrightarrow} &
 \xy 0;/r.365pc/:
(-12,6)*+{\text{\scriptsize{$2$}}}="1";(0,6)*+{\text{\scriptsize{$1$}}}="2";(12,6)*+{\text{\scriptsize{$2$}}}="3";
(-12,-6)*+{\text{\scriptsize{$1$}}}="4";(0,-6)*{\cdot}="5";(0,-7.2)*+{\text{\scriptsize{$2$}}}="";(12,-6)*+{\text{\scriptsize{$1$}}}="6";
(2,-2)*{}="7";(4,-4)*{}="8";
(1.6,-3.6)*{}="9";(2.4,-4.4)*{}="10";
{\ar"4";"1"};{\ar"1";"2"};{\ar"3";"2"};{\ar"6";"3"};{\ar"5";"6"};{\ar"5";"4"};
{\ar"2";"5"};
{\ar@{-}@/^.12pc/@[blue]"5";"7"};{\ar@{-}@/^1.55pc/@[blue]"7";"8"};{\ar@{->}@/^.12pc/@[blue]"8";"5"};
{\ar@{-}@/_.03pc/@[blue]"5";"10"};{\ar@{-}@/_.8pc/@[blue]"10";"9"};{\ar@{->}@/_.03pc/@[blue]"9";"5"};
\endxy
\\
Q & \ \ & Q^1 & \ \ & Q^2 = Q' \\
\end{array}$$
\caption{(Example \ref{Example2}.)
A maximal sequence of contractions as in (\ref{sequence}).
Each quiver is drawn on a torus.
$Q'$ is a cancellative dimer quiver with a length $1$ unit cycle, and the two blue loops are redundant generators for the dimer algebra $A' = kQ'/I'$.}
\label{example2}
\end{figure}

\section{Preliminary definitions} \label{definitions}

Throughout, $k$ is an algebraically closed field.
Given a quiver $Q$, we denote by $kQ$ the path algebra of $Q$, and by $Q_{\ell}$ the paths of length $\ell$.
The idempotent at vertex $i \in Q_0$ is denoted $e_i$, and the head and tail maps are denoted $\operatorname{h},\operatorname{t}: Q_1 \to Q_0$.
Multiplication of paths is read right to left, following the composition of maps.

\begin{Definition} \label{dimer def} \rm{ \
\begin{itemize}
 \item A \textit{dimer quiver} $Q$ is a quiver whose underlying graph $\overbar{Q}$ embeds into a real two-torus $T^2$ such that each connected component of $T^2 \setminus \overbar{Q}$ is simply connected and bounded by an oriented cycle, called a \textit{unit cycle}.\footnote{In more general contexts, the two-torus may be replaced by a compact surface (e.g., \cite{BGH, BKM}).  The dual graph of a dimer quiver is called a dimer model or brane tiling.}
The \textit{dimer algebra} $A$ of $Q$ is the quotient $kQ/I$, where $I$ is the ideal
\begin{equation} \label{I}
I := \left\langle p - q \ | \ \exists a \in Q_1 \ \text{s.t. $pa$ and $qa$ are unit cycles} \right\rangle \subset kQ,
\end{equation}
and $p,q$ are (possibly trivial) paths.
 \item $A$ and $Q$ are \textit{non-cancellative} if there are paths $p,q,r \in A$ for which $p \not = q$, and
 $$pr = qr \not = 0 \ \ \ \text{ or } \ \ \ rp = rq \not = 0;$$
 otherwise $A$ and $Q$ are \textit{cancellative}.
 \item Since $I$ is generated by binomials in the paths of $Q$, we also refer to the equivalence class $p + I$ of a path $p$ in $Q$ as a \textit{path} in $A$.
  If $p$ and $q$ are paths in $Q$ (resp.\ $A$) that are equal modulo $I$, then we will write $p \equiv q$ (resp.\ $p = q$).
\end{itemize}
}\end{Definition}

In the literature, unit cycles are typically required to have length at least $2$ or $3$.
However, we allow unit cycles to have length $1$ since such cycles may form under a cyclic contraction.
An example of a cancellative dimer algebra with a length $1$ unit cycle is given in Example \ref{Example2}.

If $a \in Q_1$ is a unit cycle and $pa$ is the complementary unit cycle containing $a$, then $p$ equals the vertex $e_{\operatorname{t}(a)}$ modulo $I$.
The case where $p$ has length $1$ leads us to introduce the following definition.

\begin{Definition} \rm{
A length $1$ path $a \in Q_1$ is an \textit{arrow} if $a$ is not equal to a vertex modulo $I$; otherwise $a$ is a \textit{pseudo-arrow}.
}\end{Definition}

The following well-known definitions are slightly modified under our distinction between arrows and length $1$ paths.

\begin{Definition} \rm{ \
\begin{itemize}
 \item A \textit{perfect matching} $D$ of $Q$ is a set of arrows such that each unit cycle contains precisely one arrow in $D$.
 \item A perfect matching $D$ is \textit{simple} if there is a cycle in $Q \setminus D$ that passes through each vertex of $Q$ (that is, the subquiver with arrow set $Q_1 \setminus D$ supports a simple $A$-module of dimension $1^{Q_0}$).
\end{itemize}
}\end{Definition}

For each perfect matching $D$, consider the map
$$n_D: Q_{\geq 0} \to \mathbb{Z}_{\geq 0}$$
defined by sending a path $p$ to the number of arrow subpaths of $p$ that are contained in $D$.
Note that if $p,q \in Q_{\geq 0}$ are paths satisfying $\operatorname{t}(p) = \operatorname{h}(q)$, then
$$n_D(pq) = n_D(p) + n_D(q).$$
Furthermore, if $p, p'$ are paths satisfying $p \equiv p'$, then $n_D(p) = n_D(p')$.
In particular, $n_D$ induces a well-defined map on the paths of $A$.

Now consider a contraction of dimer algebras $\psi: A \to A'$, with $A'$ cancellative.
Consider the polynomial ring $k\left[x_D \ | \ D \in \mathcal{S}' \right]$ generated by the simple matchings $\mathcal{S}'$ of $A'$.
To each path $p \in A'$, associate the monomial
$$\bar{\tau}(p) := \prod_{D \in \mathcal{S}'} x_D^{n_D(p)}.$$
The map $\psi$ is called a \textit{cyclic contraction} if
\begin{equation*} \label{cycle algebra}
S := k \left[ \cup_{i \in Q_0} \bar{\tau}\psi(e_iAe_i) \right] = k \left[ \cup_{i \in Q'_0} \bar{\tau}(e_iA'e_i) \right] =: S'.
\end{equation*}
In this case, we call $S$ the \textit{cycle algebra} of $A$ and $A'$, and say $\psi$ `preserves the cycle algebra'.\footnote{The uniqueness of $S$ follows from (\ref{S cong}); see Corollary \ref{last} below.}

\begin{Notation} \rm{
Denote by $\mathcal{P}$ and $\mathcal{P}'$ the set of perfect matchings of $A$ and $A'$, respectively.
For each pair of vertices $i,j \in Q_0$ and $i',j' \in Q'_0$, consider the $k$-linear maps\footnote{The monomial labelings $\bar{\eta}$, $\bar{\eta}'$, and $\bar{\tau}$ define algebra homomorphisms on $A$ and $A'$; see \cite[Section 2.1]{B1}.}
$$\bar{\eta}: e_jAe_i \to k\left[x_D \ | \ D \in \mathcal{P} \right] \ \ \ \text{ and } \ \ \ \bar{\eta}': e_{j'}A'e_{i'} \to k\left[y_D \ | \ D \in \mathcal{P}' \right],$$
defined on paths by
$$\bar{\eta}(p) := \prod_{D \in \mathcal{P}}x_D^{n_D(p)} \ \ \ \text{ and } \ \ \ \bar{\eta}'(q) := \prod_{D \in \mathcal{P}'}y_D^{n_D(q)},$$
and extended $k$-linearly.
Note that for a path $p$, we have $x_D \mid \bar{\eta}(p)$ if and only if $p$ has an arrow subpath that is contained in $D$.

For $i \in Q_0$, denote by $\sigma_i$ either a choice of unit cycle in $e_ikQe_i$, or the unique unit cycle in $e_iAe_i$.
Denote the $\bar{\eta}$, $\bar{\eta}'$, and $\bar{\tau}$ images of the unit cycles in $Q$ and $Q'$ by
$$\sigma_{\mathcal{P}} := \prod_{D \in \mathcal{P}} x_D, \ \ \ \ \ \sigma_{\mathcal{P}'} := \prod_{D \in \mathcal{P}'}y_D, \ \ \ \ \ \sigma_{\mathcal{S}'} := \prod_{D \in \mathcal{S}'}y_D.$$

Let $\pi: \mathbb{R}^2 \to T^2$ be a covering map such that for some vertex $i \in Q_0$, we have $\pi^{-1}(i) = \mathbb{Z}^2$.  
Denote by $Q^+ := \pi^{-1}(Q)$ the infinite covering quiver of $Q$.
For each path $p$ in $Q$, denote by $p^+$ the path in $Q^+$ with tail in $\left[0, 1 \right) \times \left[0, 1 \right) \subset \mathbb{R}^2$ satisfying $\pi(p^+) = p$.
}\end{Notation}

\begin{Notation} \rm{
By a \textit{cyclic subpath}, we mean a cyclic subpath that is not equal to a vertex modulo $I$.

Consider the following sets of cycles in $Q$:
\begin{itemize}
 \item Let $\mathcal{C}$ be the set of cycles in $Q$.\footnote{We caution that in the companion articles \cite{B1} - \cite{B5}, \cite{B7}, the set $\mathcal{C}$ is defined to be the set of cycles in $A$, i.e., cycles in $Q$ modulo $I$.}
 \item For $u \in \mathbb{Z}^2$, let $\mathcal{C}^u$ be the set of cycles $p \in \mathcal{C}$ such that
$$\operatorname{h}(p^+) = \operatorname{t}(p^+) + u \in Q_0^+.$$
 \item For $i \in Q_0$, let $\mathcal{C}_i$ be the set of cycles in the vertex corner ring $e_ikQe_i$.
 \item Let $\hat{\mathcal{C}}$ be the set of cycles $p \in \mathcal{C}$ such that the lift of each 
     cyclic permutation of each representative of $p+I$ does not have a cyclic subpath.
\end{itemize}
We denote the intersection $\hat{\mathcal{C}} \cap \mathcal{C}^u \cap \mathcal{C}_i$, for example, by $\hat{\mathcal{C}}^u_i$.
Note that although the lift of a cycle $p$ in $\hat{\mathcal{C}}$ has no nontrivial cyclic subpaths, $p$ itself may have cyclic subpaths.
We similarly define the set of cycles ${\mathcal{C}'}$ in $Q'$.
}\end{Notation}

\section{Cycles that avoid a perfect matching}

Let $A = kQ/I$ be a dimer algebra.
Throughout this section, set
$$\sigma := \sigma_{\mathcal{P}} \ \ \ \text{ and } \ \ \ \overbar{p} := \bar{\eta}(p).$$
By $\overbar{p} \mid \overbar{q}$, we mean that $\overbar{p}$ divides $\overbar{q}$ in $k[x_D \ | \ D \in \mathcal{P}]$.
We introduce the following.

\begin{Definition} \rm{
We say $A$ and $Q$ are \textit{cycle-nondegenerate} if each cycle of $Q$ contains an arrow that is contained in a perfect matching.
}\end{Definition}

Let $p$ be a cycle.
If $\sigma \nmid \overbar{p}$, then $p \in \hat{\mathcal{C}}$ \cite[Lemma 4.8.3]{B1}; and if $\hat{\mathcal{C}}^u_i \not = \emptyset$ for each $i \in Q_0$ and $u \in \mathbb{Z}^2$, then the converse holds \cite[Proposition 4.20.1]{B1}.
In this section, we show that if $Q$ is cycle-nondegenerate, then for each $u \in \mathbb{Z}^2$ there is a cycle $p \in \hat{\mathcal{C}}^u$ such that $\sigma \nmid \overbar{p}$.

\begin{Lemma} \label{fromBa}
Suppose $p$ and $q$ are paths in $Q$ for which
$$\operatorname{t}(p^+)  = \operatorname{t}(q^+) \ \ \ \text{ and } \ \ \ \operatorname{h}(p^+) = \operatorname{h}(q^+).$$
Then there is an $m,n \geq 0$ such that
\begin{equation*} \label{first}
p \sigma_{\operatorname{t}(p)}^m \equiv q \sigma_{\operatorname{t}(p)}^n \ \ \ \text{ and } \ \ \ \overbar{p} = \overbar{q} \sigma^{n-m}.
\end{equation*}
Furthermore, if $p^+$ is a cycle in $Q^+$, then there is an $m \geq 0$ such that
\begin{equation*} \label{third}
\overbar{p} = \sigma^m.
\end{equation*}
\end{Lemma}

\begin{proof}
The claims hold respectively by \cite[Lemmas 4.3.1, 4.3.2, and 4.8.1]{B1}.
\end{proof}

For $u \in \mathbb{Z}^2$, set
$$\mathcal{C}^{u,n} := \left\{ p \in \mathcal{C}^u \ | \ \sigma^n \mid \overbar{p} \ \text{ and } \ \sigma^{n+1} \nmid \overbar{p} \right\}.$$
Let $\mu = \mu(u) \geq 0$ be the smallest integer for which $\mathcal{C}^{u,\mu} \not = \emptyset$.
The main purpose of this section is to show that for every $u \in \mathbb{Z}^2$, we have $\mu(u) = 0$.

%
\begin{Lemma} \label{partition}
Let $p$ and $q$ be subpaths of cycles in $\mathcal{C}^{u,\mu}$.
If
\begin{equation} \label{tails}
\operatorname{t}(p^+) = \operatorname{t}(q^+) \ \ \ \text{ and } \ \ \ \operatorname{h}(p^+) = \operatorname{h}(q^+),
\end{equation}
then $\overbar{p} = \overbar{q}$.
\end{Lemma}

\begin{proof}
Set $p_1 := p$ and $q_1 := q$, and suppose (\ref{tails}) holds.
Since $p_1, q_1$ are subpaths of cycles in $\mathcal{C}^{u,\mu}$, there are (possibly trivial) paths $p_2,q_2$ such that $p_2p_1$ and $q_2q_1$ are cycles in $\mathcal{C}^{u,\mu}$.
Since $p_2p_1$ is a cycle in $\mathcal{C}^u$, we have
$$\operatorname{t}(p_1^+) + u = \operatorname{h}(p_2^+) = \operatorname{t}(q_1^+) + u = \operatorname{h}(q_2^+), \ \ \ \ \ \ \operatorname{h}(p_1^+) = \operatorname{t}(p_2^+) = \operatorname{h}(q_1^+) = \operatorname{t}(q_2^+).$$

Let $n_1$ and $n_2$ be the respective maximum powers of $\sigma$ that divide $\overbar{q}_1$ and $\overbar{p}_1$.
Assume to the contrary that $n_1 < n_2$.
Set $s := n_2 - n_1 > 0$.
By Lemma \ref{fromBa}, we have
\begin{equation} \label{p1}
\overbar{p}_1 = \overbar{q}_1 \sigma^{n_2 - n_1} = \overbar{q}_1 \sigma^s.
\end{equation}
Since $p_2p_1$ is in $\mathcal{C}^{u,\mu}$, there is a monomial $g \in k[x_D \ | \ D \in \mathcal{P} ]$, not divisible by $\sigma$, such that
$$\overbar{p_2p_1} = g\sigma^{\mu}.$$
Whence
$$\overbar{p_2q_1} = \overbar{p}_2 \, \overbar{q}_1 = \overbar{p}_2 \, \overbar{p}_1 \sigma^{-s} = \overbar{p_2p_1} \sigma^{-s} = g \sigma^{\mu -s}.$$
Thus the maximum power of $\sigma$ that divides $\overbar{p_2q_1}$ is $\mu -s < \mu$.
But the cycle $p_2q_1$ is in $\mathcal{C}^u$, contrary to the minimality of $\mu$.
Therefore $n_1 \geq n_2$.
A similar argument shows that $n_1 \leq n_2$, and so $n_1 = n_2$.
Consequently, $\overbar{p}_1 = \overbar{q}_1$ by (\ref{p1}).
\end{proof}

\begin{Lemma} \label{partition2}
Let $p = p_m \cdots p_2p_1$ be a cycle in $\mathcal{C}^u$ such that each $p_i$ is a subpath of a cycle in $\mathcal{C}^{u,\mu}$.
Then $p$ is in $\mathcal{C}^{u,\mu}$.
\end{Lemma}

\begin{proof}
By induction, it suffices to suppose $p = p_2p_1$.
Since $p_1, p_2$ are subpaths of cycles in $\mathcal{C}^{u,\mu}$, there are paths $q_1,q_2$ such that $q_2p_1$ and $p_2q_1$ are cycles in $\mathcal{C}^{u,\mu}_{\operatorname{t}(p_1)}$.
By Lemma \ref{partition}, we have $\overbar{p}_2 = \overbar{q}_2$ (and $\overbar{p}_1 = \overbar{q}_1$).
Thus
$$\overbar{p_2p_1} = \overbar{p}_2 \, \overbar{p}_1 = \overbar{q}_2 \, \overbar{p}_1 = \overbar{q_2p_1}.$$
Therefore $p_2p_1$ is in $\mathcal{C}^{u,\mu}$.
\end{proof}

\begin{Lemma} \label{yoohooo}
If $Q$ has a perfect matching, then there is containment $\mathcal{C}^{u,\mu} \subseteq \hat{\mathcal{C}}^u$.
\end{Lemma}

\begin{proof}
Suppose $p \in \mathcal{C}^u \setminus \hat{\mathcal{C}}^u$.
Then there is a cyclic permutation $q$ of a representative of $p + I$ such that $q^+$ has a cyclic subpath $r^+$.
Clearly, $\overbar{q} = \overbar{p}$.

By Lemma \ref{fromBa}, $\overbar{r} = \sigma^n$ for some $n \geq 1$.
Furthermore, $\sigma \not = 1$ since the set of perfect matchings $\mathcal{P}$ is nonempty.
But then upon removing $r$ from $q$ we find that $q$, whence $p$, is not in $\mathcal{C}^{u,\mu}$.
\end{proof}

Fix $u \in \mathbb{Z}^2 \setminus 0$.

\begin{Construction} \label{Q} \rm{
Given a path $p$, denote by $[p]$ the subquiver whose vertex and arrow sets are the vertex and arrow subpaths of $p$.
To show that $\mu = 0$, consider a maximal chain of subquivers
\begin{equation} \label{chain2}
Q^0 \subseteq Q^1 \subseteq \cdots \subseteq Q^N \subseteq Q,
\end{equation}
where
$$Q^{0} := \bigcup_{p \in \mathcal{C}^{u,\mu}}[p],$$
and for $1 \leq m \leq N$, there is a particular path $s_m$ such that
$$Q^m := Q^{m-1} \cup [s_m].$$
The sequence of paths $s_1, s_2, \ldots, s_N$ is defined inductively as follows.

Fix $n \geq 1$, and suppose that the sequence $s_1, s_2, \ldots, s_{n-1}$ has been defined.
Consider the set $\mathcal{S}^n$ of paths $s \in Q_{\geq 1}$ for which
\begin{enumerate}[(i)]
 \item the endpoints of $s$ are in $Q^{n-1}_0$, and no other subpath of $s$ lies in $Q^{n-1}$;
 \item if there is a path $p^+$ in $(Q^{n-1})^+$ with endpoints $\operatorname{t}(s^+)$ and $\operatorname{h}(s^+)$, such that the interior of the compact region in $\mathbb{R}^2$ bounded by $p^+$ and $s^+$ contains no arrows in $(Q^{n-1})^+$, then
$$\operatorname{t}(s^+) = \operatorname{t}(p^+) \ \ \ \text{ and } \ \ \ \operatorname{h}(s^+) = \operatorname{h}(p^+).$$
\end{enumerate}
(Condition (ii) says, loosely, that $s^+$ `flows in the same direction' as all the paths in $(Q^{n-1})^+$.)

We now define an order $<$ on $\mathcal{S}^n$.
We say $s \in \mathcal{S}^n$ satisfies ($\star$) if there are paths $p_1,p_2,q_1,q_2 \in Q^{n-1}_{\geq 0}$ such that
$$\operatorname{t}((p_2s_{n-1}p_1)^+) = \operatorname{t}((q_2sq_1)^+), \ \ \ \ \operatorname{h}((p_2s_{n-1}p_1)^+) = \operatorname{h}((q_2sq_1)^+),$$
and for some non-negative integer $\ell \geq 0$,
$$\overbar{p_2s_{n-1}p_1} = \overbar{q_2sq_1} \sigma^{\ell}.$$
Given $s,t \in \mathcal{S}^n$, we declare $s < t$ if
\begin{enumerate}
 \item $s$ satisfies ($\star$) and $t$ does not; or
 \item $s$ satisfies ($\star$) iff $t$ satisfies ($\star$); and $\overbar{s} \mid \overbar{t}$.
\end{enumerate}
Finally, let $s_n$ be a minimal element of $\mathcal{S}^n$ with respect to $<$.
}\end{Construction}

We now use the chain of subquivers (\ref{chain2}) to construct a perfect matching $D$ of $Q$ that avoids the cycles in $\mathcal{C}^{u,\mu}$.

\begin{Construction} \label{D} \rm{
For $n \geq 0$, denote by $D^n$ the set of arrows $a$ in $Q_1 \setminus Q^n_1$ such that
\begin{enumerate}
 \item the tail $\operatorname{t}(a)$ is in $Q^n_0$;
 \item $a$ is not in a unit cycle that contains an arrow in $D^m$ for $0 \leq m < n$; and
 \item $a$ is the unique arrow in the unit cycles containing $a$ that satisfies (1) and (2).
\end{enumerate}
Note that $D^n$ is empty for sufficiently large $n$.
Set
$$D := \bigcup_{n \geq 0} D^{n} \subset Q_1.$$
}\end{Construction}

\begin{figure}
$$\begin{array}{ccccccc}
\xy 0;/r.365pc/:
(0,0)*{\cdot}="1";(-8,-14)*{}="2";
(-8,-6)*{\cdot}="3";(-8,6)*{\cdot}="5";(-8,14)*{}="6";(8,0)*{}="9";
(4,4)*{}="20";(4,-4)*{}="21";
{\ar@[red]"2";"3"};{\ar@[red]"3";"5"};{\ar@[red]"5";"6"};
{\ar_{r_1}"3";"1"};{\ar_{r_2}"1";"5"};
{\ar@{-}@/^.5pc/@[brown]"1";"20"};{\ar@{-}@/^.5pc/@[brown]_{\sigma_i}"20";"9"};{\ar@{-}@/^.5pc/@[brown]"9";"21"};{\ar@/^.5pc/@[brown]"21";"1"};
\endxy
& \ \ \ &
\xy 0;/r.365pc/:
(0,0)*{\cdot}="1";(-8,-14)*{}="2";
(-8,0)*{\cdot}="4";(-8,14)*{}="6";(8,0)*{\cdot}="9";
(8,-14)*{}="7";(8,14)*{}="11";
(0,8)*{}="12";
(4,4)*{}="20";(-4,4)*{}="21";
{\ar@[red]"2";"4"};{\ar@[red]"4";"6"};
{\ar@[red]"7";"9"};{\ar@[red]"9";"11"};
{\ar_{r_1}"4";"1"};{\ar_{r_2}"1";"9"};
{\ar@{-}@[brown]@/_.5pc/"1";"20"};{\ar@{-}@[brown]@/_.5pc/^{\sigma_i}"20";"12"};{\ar@{-}@[brown]@/_.5pc/"12";"21"};{\ar@[brown]@/_.5pc/"21";"1"};
\endxy
& \ \ \ &
\xy 0;/r.365pc/:
(0,0)*{\cdot}="1";(-8,-14)*{}="2";
(-8,-6)*{\cdot}="3";(-8,6)*{\cdot}="5";(-8,14)*{}="6";(8,0)*{\cdot}="9";
{\ar@[red]"2";"3"};{\ar@[red]"3";"5"};{\ar@[red]"5";"6"};
{\ar@[brown]_{b}"1";"3"};{\ar@[brown]_{a}"5";"1"};
{\ar@[brown]@/_1.2pc/_{c}"3";"9"};{\ar@[brown]@/_1.2pc/_d"9";"5"};
\endxy
& \ \ \ &
\xy 0;/r.365pc/:
(-8,-6)*{}="3";(-8,0)*{\cdot}="4";(-8,6)*{}="5";(8,-6)*{}="8";(8,0)*{\cdot}="9";(8,6)*{}="10";
(0,6)*{\cdot}="12";(0,-6)*{\cdot}="14";
(-8,14)*{}="6";(8,14)*{}="11";(-8,-14)*{}="2";(8,-14)*{}="7";
{\ar@[red]@{-}"2";"3"};{\ar@[red]^{}"3";"4"};{\ar@[red]@{-}^{}"4";"5"};
{\ar@[red]"5";"6"};
{\ar@{-}@[red]"7";"8"};{\ar@[red]^{}"8";"9"};{\ar@[red]@{-}^{}"9";"10"};
{\ar@[red]"10";"11"};
{\ar@[brown]^{a}"4";"14"};{\ar@[brown]^{b}"14";"9"};{\ar@[brown]^{c}"9";"12"};{\ar@[brown]^{d}"12";"4"};
\endxy
\\
(2.a) & & (2.a) & & (2.b) & & (2.b)
\end{array}$$
\caption{Cases for Proposition \ref{Clock}.
Each quiver is drawn on the torus, the unit cycle $\sigma_i$ is drawn in brown, and the red arrows are paths in $Q^{n-1}$.
In (2.a), $r$ factors into paths $r = r_2r_1$, with $r_1,r_2 \in Q_{\geq 1}$.
In (2.b), $\sigma_i$ factors into paths $dcba$, with $b,d \in Q_{\geq 0}$.
Irrelevant vertices and arrows are omitted.}
\label{clock}
\end{figure}

\begin{Proposition} \label{Clock}
Suppose $Q$ is cycle-nondegenerate.
Then $D$ is a perfect matching of $Q$.
\end{Proposition}

\begin{proof}
By conditions (2) and (3) in Construction \ref{D}, each unit cycle of $Q$ contains at most one arrow in $D$.
Thus, to show that $D$ is a perfect matching of $Q$, it suffices to show that each unit cycle also contains at least one arrow in $D$.

Let $N \geq 1$ be the minimum integer for which $Q^N = Q^{N+1}$.
Assume to the contrary that there is a unit cycle $\sigma_i$ that does not have an arrow subpath in $D$.
Then one of the following cases holds.
\begin{enumerate}
 \item $\sigma_i$ is a unit cycle of $Q^0$.
 \item There is an $n \geq 1$ such that $\sigma_i$ is a unit cycle of $Q^n$ that is not wholly contained in $Q^{n-1}$, and $\sigma_i$ does not contain precisely one arrow subpath in $Q^n_1 \setminus Q^{n-1}_1$ with tail in $Q_0^{n-1}$.
 \begin{enumerate}
  \item $\sigma_i$ has no arrow subpath $a \in Q^n_1 \setminus Q^{n-1}_1$ for which $\operatorname{t}(a) \in Q^{n-1}_0$.
  \item $\sigma_i$ has at least two arrow subpaths $a,c \in Q^n_1 \setminus Q^{n-1}_1$ with $\operatorname{t}(a),\operatorname{t}(c) \in Q^{n-1}_0$.
 \end{enumerate}
 \item $\sigma_i$ has at least two arrow subpaths $a,c \in Q_1 \setminus Q^N_1$ with $\operatorname{t}(a), \operatorname{t}(c) \in Q^N_0$.
\end{enumerate}

We claim that each case is not possible.

\textit{Case (1):} $Q^0$ cannot contain a unit cycle of $Q$ by Lemmas \ref{partition2} and \ref{yoohooo}.

Recall that $s_n$ is a minimal path in $\mathcal{S}^n$, and $Q^n = Q^{n-1} \cup [s_n]$.
The following two subcases are shown in Figure \ref{clock}.

\textit{Case (2.a):} By assumption, $\sigma_i$ is a unit cycle of $Q^n$ that does not meet $Q^{n-1}$.
Thus $\sigma_i$ may be formed from subpaths of $s_n$.
But then $s^+_n$ has a proper nontrivial cyclic subpath $t^+$.
Let $r$ be the path obtained by omitting $t$ from $s_n$.
Since $t^+$ is a cycle and $s_n \in \mathcal{S}^n$, we have $r \in \mathcal{S}^n$.

Since $r$ is obtained by omitting $t$ from $s_n$, we have $\overbar{r} \mid \overbar{s}_n$.
Thus, if $s_n$ satisfies ($\star$), then $r$ satisfies ($\star$).
Whence $r < s_n$.
Furthermore, since $Q$ is cycle-nondegenerate, $\overbar{t} \not = 1$.
In particular, $\overbar{s}_n \nmid \overbar{r}$.
Therefore $s_n \not < r$.
But this contradicts the minimality of $s_n$ in $\mathcal{S}^n$ with respect to $<$.

\textit{Case (2.b):} Since $a,c$ are arrows in $Q^n_1 \setminus Q^{n-1}_1$, both $a$ and $c$ are subpaths of $s_n$.
By assumption, $\operatorname{t}(a), \operatorname{t}(c) \in Q^{n-1}_0$, and so $s_n$ meets $Q^{n-1}$ at a trivial subpath other than its endpoints.
But then $s_n$ is not in $\mathcal{S}^n$, contrary to our choice of $s_n$.

\textit{Case (3):} Let $\mathcal{A}$ be the set of arrow subaths of $\sigma_i$ in $Q_1 \setminus Q^N_1$.
For $i \in Q_0$, denote by $m(i)$ the minimum integer for which $i \in Q^{m(i)}_0$.
Let $a \in \mathcal{A}$ be such that
$$m(\operatorname{t}(a)) = \operatorname{min}\{ m(\operatorname{t}(c)) \ | \ c \in \mathcal{A} \}.$$

Since $|\mathcal{A}| \geq 2$, each arrow in $\mathcal{A}$ is not in $D$, by condition (3) of Construction \ref{D}.
In particular, $a \not \in D$.
Thus there is an arrow $c \in \mathcal{A} \setminus \{a\}$ and a path $t^+$ with endpoints $\operatorname{t}(a^+)$ and $\operatorname{t}(c^+)$, such that $t$ is a subpath of $s_{m(\operatorname{t}(a))}$.
Let $b,d \in Q_{\geq 0}$ be paths such that $\sigma_i = dcba$.

Without loss of generality, suppose
$$\operatorname{t}(t^+) = \operatorname{t}(a^+) \ \ \ \text{ and } \ \ \ \operatorname{h}(t^+) = \operatorname{t}(c^+).$$
Then by Lemma \ref{fromBa}, there is an $\ell \in \mathbb{Z}$ such that
\begin{equation} \label{doc!}
\overbar{t} = \overbar{ba} \, \sigma^{\ell}.
\end{equation}
Furthermore, $\ell \geq -1$ since $ba$ is a subpath of a unit cycle.

If $\ell = -1$, then $\overbar{dc} = 1$ and $\overbar{t} = 1$, since $dcba$ is a unit cycle.
In particular, the cycle $tdc$ satisfies
$$\overbar{tdc} = \overbar{t} \, \overbar{dc} = 1.$$
But then $Q$ is cycle-degenerate, contrary to assumption.
Thus $\ell \geq 0$.
Therefore (\ref{doc!}) implies that $ba$ is in $Q^N$ since $t$ is in $Q^N$, by property ($\star$).
However, it then follows that $a$ is in $Q^N$, contrary to assumption.

We have shown that each case is not possible.
Therefore each unit cycle contains precisely one arrow in $D$, and so $D$ is a perfect matching of $Q$.
\end{proof}

\begin{Corollary} \label{dude}
Suppose $Q$ is cycle-nondegenerate.
Then for each $u \in \mathbb{Z}^2$, the set of cycles $\mathcal{C}^{u,0}$ is nonempty.
\end{Corollary}

\begin{proof}
If $u = 0 \in \mathbb{Z}^2$, then each vertex of $Q$ is in $\mathcal{C}^{u,0}$.

So let $u \in \mathbb{Z}^2 \setminus 0$, and set $\mu := \mu(u)$.
By Proposition \ref{Clock}, there is a perfect matching $D$ such that each cycle $p$ in $\mathcal{C}^{u,\mu}$ is supported on $Q \setminus D$.
In particular, $x_D \nmid \overbar{p}$.
Whence, $\sigma \nmid \overbar{p}$.
Therefore $\mu = 0$.
\end{proof}

\begin{Example} \label{examplenested} \rm{
We demonstrate the construction of a perfect matching $D$ that avoids the cycles in $\mathcal{C}^{(0,1)}$.
Consider the subquiver $Q^+$ of a dimer quiver shown in Figure \ref{examplenestedfigure}.
\begin{enumerate}[($i$)]
 \item $Q^0$ is formed from the red arrows, and $D^0$ consists of the three brown arrows.
 \item We may choose $s_1$ to be the path formed from the red arrows, and $s_2$ to be the path formed from the blue arrows.
 Then $D^1 = \emptyset$, and $D^2$ consists of the single brown arrow.
 \item Alternatively, we may choose $s_1$ to be the cycle at $i$ formed from the red arrows; $s_2$ to be the purple arrow; and $s_3$ to be the path formed from the blue arrows.
Then $D^1 = D^2 = \emptyset$, and $D^3$ consists of the single brown arrow.
Note that there are other choices for the paths $s_1, s_2, s_3$ as well.
 \item With any choice of paths, the union $D = \cup D^n$ consists of the four brown arrows, and is a perfect matching of $Q$.
\end{enumerate}
}\end{Example}

\begin{figure}
$$\begin{array}{ccccccc}
\xy 0;/r.275pc/:
(0,-16)*{\cdot}="1";(-12,-4)*{}="2";(-12,4)*{}="3";
(0,16)*{\cdot}="4";(12,-4)*{}="5";(12,4)*{}="6";
(-4,0)*{\cdot}="11";(4,0)*{\cdot}="12";
(12,0)*{\cdot}="7";(-12,0)*{\cdot}="8";
(0,4)*{\cdot}="9";(0,-4)*{\cdot}="10";
{\ar@/^.9pc/@{-}@[red]"1";"2"};{\ar@[red]"2";"8"};{\ar@{-}@[red]"8";"3"};{\ar@/^.9pc/@[red]"3";"4"};
{\ar@/_.9pc/@{-}@[red]"1";"5"};{\ar@[red]"5";"7"};{\ar@{-}@[red]"7";"6"};{\ar@/_.9pc/@[red]"6";"4"};
{\ar@/_.4pc/"9";"8"};{\ar@/_.4pc/@[brown]"8";"10"};{\ar@/^.4pc/"9";"7"};{\ar@/^.4pc/@[brown]"7";"10"};
{\ar"11";"9"};{\ar"10";"11"};{\ar"10";"12"};{\ar"12";"9"};
{\ar@[brown]"4";"9"};{\ar"9";"10"};{\ar"10";"1"};
\endxy
& \ \ &
\xy 0;/r.275pc/:
(0,-16)*{\cdot}="1";(-12,-4)*{}="2";(-12,4)*{}="3";
(0,16)*{\cdot}="4";(12,-4)*{}="5";(12,4)*{}="6";
(-4,0)*{\cdot}="11";(4,0)*{\cdot}="12";
(12,0)*{\cdot}="7";(-12,0)*{\cdot}="8";
(0,4)*{\cdot}="9";(0,-4)*{\cdot}="10";
{\ar@/^.9pc/@{-}"1";"2"};{\ar"2";"8"};{\ar@{-}"8";"3"};{\ar@/^.9pc/"3";"4"};
{\ar@/_.9pc/@{-}"1";"5"};{\ar"5";"7"};{\ar@{-}"7";"6"};{\ar@/_.9pc/"6";"4"};
{\ar@/_.4pc/"9";"8"};{\ar@/_.4pc/@[red]"8";"10"};{\ar@/^.4pc/@[red]"9";"7"};{\ar@/^.4pc/"7";"10"};
{\ar@[red]"11";"9"};{\ar@[red]"10";"11"};{\ar@[blue]"10";"12"};{\ar@[blue]"12";"9"};
{\ar"4";"9"};{\ar@[brown]"9";"10"};{\ar"10";"1"};
\endxy
& \ \ &
\xy 0;/r.275pc/:
(0,-16)*{\cdot}="1";(-12,-4)*{}="2";(-12,4)*{}="3";
(0,16)*{\cdot}="4";(12,-4)*{}="5";(12,4)*{}="6";
(-4,0)*{\cdot}="11";(4,0)*{\cdot}="12";
(12,0)*{\cdot}="7";(-12,0)*+{\text{\scriptsize{$i$}}}="8";
(0,4)*{\cdot}="9";(0,-4)*{\cdot}="10";
{\ar@/^.9pc/@{-}"1";"2"};{\ar"2";"8"};{\ar@{-}"8";"3"};{\ar@/^.9pc/"3";"4"};
{\ar@/_.9pc/@{-}"1";"5"};{\ar"5";"7"};{\ar@{-}"7";"6"};{\ar@/_.9pc/"6";"4"};
{\ar@/_.4pc/@[red]"9";"8"};{\ar@/_.4pc/@[red]"8";"10"};{\ar@/^.4pc/@[violet]"9";"7"};{\ar@/^.4pc/"7";"10"};
{\ar@[red]"11";"9"};{\ar@[red]"10";"11"};{\ar@[blue]"10";"12"};{\ar@[blue]"12";"9"};
{\ar"4";"9"};{\ar@[brown]"9";"10"};{\ar"10";"1"};
\endxy
& \ \ &
\xy 0;/r.275pc/:
(0,-16)*{\cdot}="1";(-12,-4)*{}="2";(-12,4)*{}="3";
(0,16)*{\cdot}="4";(12,-4)*{}="5";(12,4)*{}="6";
(-4,0)*{\cdot}="11";(4,0)*{\cdot}="12";
(12,0)*{\cdot}="7";(-12,0)*{\cdot}="8";
(0,4)*{\cdot}="9";(0,-4)*{\cdot}="10";
{\ar@/^.9pc/@{-}"1";"2"};{\ar"2";"8"};{\ar@{-}"8";"3"};{\ar@/^.9pc/"3";"4"};
{\ar@/_.9pc/@{-}"1";"5"};{\ar"5";"7"};{\ar@{-}"7";"6"};{\ar@/_.9pc/"6";"4"};
{\ar@/_.4pc/"9";"8"};{\ar@/_.4pc/@[brown]"8";"10"};{\ar@/^.4pc/"9";"7"};{\ar@/^.4pc/@[brown]"7";"10"};
{\ar"11";"9"};{\ar"10";"11"};{\ar"10";"12"};{\ar"12";"9"};
{\ar@[brown]"4";"9"};{\ar@[brown]"9";"10"};{\ar"10";"1"};
\endxy
\\ \\
(i) & & (ii) & & (iii) & & (iv) \\
\end{array}$$
\caption{(Example \ref{examplenested}.)
A subquiver demonstrating the construction of a perfect matching $D$ that avoids the cycles in $\mathcal{C}^{u,0}$.}
\label{examplenestedfigure}
\end{figure}

\section{Proof of main theorem}

Let $A = kQ/I$ be a dimer algebra.
To prove our main theorem, we introduce the following.

\begin{Definition} \rm{ \
\begin{itemize}
\item We say two perfect matchings $D, D'$ of $Q$ are \textit{equivalent} if for each cycle $p$, we have
$$n_D(p) = n_{D'}(p).$$
 \item A perfect matching is \textit{rigid} if it is not equivalent to another perfect matching.
 \item An arrow $a$ is \textit{nonrigid} if every perfect matching that contains $a$ is equivalent to a perfect matching that does not contain $a$; otherwise $a$ is \textit{rigid}.
\end{itemize}
}\end{Definition}

Note that an arrow is nonrigid if it is not contained in any perfect matching, and it is rigid if it is contained in a rigid perfect matching.
In Propostion \ref{rigid2}, we will show that rigidity characterizes simple matchings.

\begin{Example} \rm{
We give an example of equivalent perfect matchings.
Suppose $D$ is a perfect matching of $Q$ and $Q \setminus D$ has a source at vertex $i$.
Let $\alpha$ and $\beta$ be the set of arrows of $Q$ with head at $i$ and tail at $i$, respectively.
Then $\alpha \subseteq D$.
Whence $\beta \cap D = \emptyset$, since each unit cycle of $Q$ contains precisely one arrow in $D$.
Let $D'$ be the perfect matching obtained from $D$ by replacing the subset $\alpha$ with the set $\beta$.
Then $D$ and $D'$ are equivalent perfect matchings.
}\end{Example}

\begin{Proposition} \label{cycle-nondegenerate}
Suppose $Q$ is cycle-nondegenerate.
Let $Q'$ be the quiver obtained from $Q$ by contracting a single nonrigid arrow $\delta \in Q_1$.
Then
\begin{enumerate}
 \item no cycle of $Q$ is contracted to a vertex; and
 \item $Q'$ is a cycle-nondegenerate dimer quiver.
\end{enumerate}
\end{Proposition}

\begin{proof}
In the following, set $\sigma := \sigma_{\mathcal{P}}$ and $\sigma' := \sigma_{\mathcal{P}'}$.
Denote by $\psi: kQ \to kQ'$ the $k$-linear map defined by contracting $\delta$.
For paths $s \in Q_{\geq 0}$ and $t \in Q'_{\geq 0}$, set
$$\overbar{s} := \bar{\eta}(s) \ \ \ \text{ and } \ \ \ \overbar{t} := \bar{\eta}'(t).$$

(i) We first claim that $Q'$ is a dimer quiver.
It suffices to show that no cycle in $Q$ contracts to a vertex under $\psi$.
Assume to the contrary that there is a (nontrivial) cycle $p$ for which $\psi(p)$ is a vertex.

Since $\delta$ is the only contracted arrow, we have $p = \delta$.
Furthermore, $\delta$ is not equal to a vertex modulo $I$ since it is an arrow, rather than a pseudo-arrow.
Thus $\delta$ is contained in a perfect matching of $Q$, since $Q$ is cycle-nondegenerate.
But $\delta$ is a cycle of length $1$.
Therefore $\delta$ is rigid, contrary to assumption.

(ii) We claim that if $D$ is a perfect matching of $Q$ and $x_D \nmid \overbar{\delta}$ (that is, $\delta$ is not contained in $D$), then $\psi(D)$ is a perfect matching of $Q'$.

Consider a unit cycle $\sigma'_j$ of $Q'$.
Each unit cycle of $Q'$ admits a $\psi$-preimage that is a unit cycle of $Q$.\footnote{We note that a unit 2-cycle may form from contracting $\delta$, and the two arrows it is composed of may be redundant generators of $A'$ \cite[Lemma 4.6]{B1}.
If these arrows are removed, then there would be a unit cycle of $Q'$ that does not have $\psi$-preimage which is a unit cycle of $Q$.}
Thus, there is a unit cycle $\sigma_i$ of $Q$ such that $\psi(\sigma_i) = \sigma'_j$.
Since $D$ is a perfect matching, there is precisely one arrow subpath $a$ of $\sigma_i$ that is in $D$.
Furthermore, since $D$ does not contain $\delta$, we have $a \not = \delta$.
Whence $\psi(a)$ is an arrow in $\sigma'_j$.
Thus there is precisely one arrow subpath of $\sigma'_j$, namely $\psi(a)$, that is in $\psi(D)$.
Therefore $\psi(D)$ is a perfect matching of $Q'$.

(iii) We now claim that $Q'$ is cycle-nondegenerate.
Let $q$ be a cycle in $Q'$.
We want to show that $\overbar{q} \not = 1$.

(iii.a) First suppose that $q$ has a $\psi$-preimage $p$ which is a cycle.
Since $Q$ is cycle-nondegenerate, there is a perfect matching $D \in \mathcal{P}$ for which $x_D \mid \overbar{p}$.

If $x_D \nmid \overbar{\delta}$, then $\psi(D)$ is a perfect matching of $Q'$ by Claim (ii).
Thus $y_{\psi(D)} \mid \overbar{q}$.

So suppose $x_D \mid \overbar{\delta}$.
Since $\delta$ is nonrigid, there is a perfect matching $D' \in \mathcal{P}$ equivalent to $D$ that does not contain $\delta$.
In particular, $x_{D'} \mid \overbar{p}$.
Furthermore, $\psi(D')$ is a perfect matching of $Q'$, again by Claim (ii).
Thus $y_{\psi(D')} \mid \overbar{q}$.

Therefore in either case, $\overbar{q} \not = 1$.

(iii.b) Finally, suppose that $q$ does not have a $\psi$-preimage which is a cycle (in particular, $q$ may not admit any $\psi$-preimage).
Assume to the contrary that $\overbar{q} = 1$.

Since $\psi$ contracts a single arrow, there is a cyclic permutation $q'$ of $q$, and a path $s$ in $Q$, such that $\psi(s) = q'$; see Figure \ref{phew}.
Then
$$\overbar{\psi(s)} = \overbar{q'} = \overbar{q} = 1.$$

Let $t$ be a path for which $\delta t$ is a unit cycle; then $ts$ is a cycle in $Q$.
Let $u \in \mathbb{Z}^2$ be such that $ts \in \mathcal{C}^u$.
By Corollary \ref{dude}, there is a cycle $r$ in $\mathcal{C}^{u,0}$.
By Lemma \ref{fromBa}, there is an $m \in \mathbb{Z}$ such that
\begin{equation} \label{sigmam}
\sigma^m \overbar{r} = \overbar{ts}.
\end{equation}
Since $r \in \mathcal{C}^{u,0}$, we have $\sigma \nmid \overbar{r}$.
Whence $m \geq 0$.
Furthermore,
\begin{equation} \label{neat-o}
\sigma'^m \overbar{\psi(r)} \stackrel{\textsc{(i)}}{=} \overbar{\psi(ts)}
\stackrel{\textsc{(ii)}}{=} \overline{\psi(t)\psi(s)} = \overbar{\psi(t)} \, \overbar{\psi(s)}
\stackrel{\textsc{(iii)}}{=} \overbar{\psi(t)}
\stackrel{\textsc{(iv)}}{=} \sigma',
\end{equation}
where (\textsc{i}) holds by (\ref{sigmam}); (\textsc{ii}) holds since $ts \not = 0$; (\textsc{iii}) holds since $\overbar{\psi(s)} = 1$; and (\textsc{iv}) holds since $\psi(t)$ is a unit cycle.

Now $\overbar{r} \not = 1$ since $r$ is a cycle and $Q$ is cycle-nondegenerate.
Thus $\overbar{\psi(r)} \not = 1$ by Claim (iii.a), with $r$ and $\psi(r)$ in place of $p$ and $q$, respectively.
But then (\ref{neat-o}) implies $m = 0$.
Therefore (\ref{sigmam}) implies
\begin{equation} \label{ts}
\overbar{r} = \overline{ts}.
\end{equation}

Since $r \in \mathcal{C}^{u,0}$, there is a perfect matching $D \in \mathcal{P}$ such that $x_D \nmid \overbar{r}$.
Thus $x_D \nmid \overbar{ts}$ by (\ref{ts}).
Whence $x_D \nmid \overbar{t}$.
Consequently, $x_D \mid \overbar{\delta}$ since $\delta t$ is a unit cycle.
Furthermore, $\delta$ is nonrigid, and so there is a perfect matching $D' \in \mathcal{P}$ equivalent to $D$ such that $x_{D'} \nmid \overbar{\delta}$.
Thus $x_{D'} \mid \overbar{t}$, again since $\delta t$ is a unit cycle.
Hence, $n_{D'}(t) \geq 1$.
Therefore, since $D$ and $D'$ are equivalent, we have
$$0 = n_D(r) = n_D(ts) = n_{D'}(ts) = n_{D'}(t) + n_{D'}(s) \geq 1.$$
But this is not possible, proving our claim.
\end{proof}

\begin{figure}
$$\xy 0;/r.3pc/:
(-5,-10)*{\cdot}="1";(-5,5)*{\cdot}="2";(-5,10)*{\cdot}="3";(5,-10)*{\cdot}="4";(5,10)*{\cdot}="5";
{\ar_s"1";"2"};{\ar@[green]^{\delta}"3";"2"};{\ar@{..>}"3";"2"};{\ar@/^1.5pc/^t"2";"3"};
{\ar_{r}"4";"5"};
\endxy$$
\caption{Setup for Claim (iii.b) in the proof of Proposition \ref{cycle-nondegenerate}, drawn on the cover $Q^+$.
The path $\delta$ is an arrow, $t\delta$ is a unit cycle, and $r$ is a cycle in $\mathcal{C}^{u,0}$.}
\label{phew}
\end{figure}

Let $Q$ be a cycle-nondegenerate dimer quiver.
By Proposition \ref{cycle-nondegenerate}, we may consider a maximal sequence of $k$-linear maps of dimer path algebras
\begin{equation} \label{sequence}
kQ \stackrel{\psi_0}{\longrightarrow} kQ^1 \stackrel{\psi_1}{\longrightarrow} kQ^2 \stackrel{\psi_2}{\longrightarrow} \cdots \stackrel{\psi_m}{\longrightarrow} kQ',
\end{equation}
where $\psi_j$ contracts one nonrigid arrow $\delta_j$ of $Q^j$.
We claim that the composition
$$\psi := \psi_m \cdots \psi_0: kQ \to kQ'$$
induces a cyclic contraction of dimer algebras
$$\psi: A = kQ/I \to A' = kQ'/I'.$$
In particular, $A'$ is cancellative and the cycle algebra is preserved.

\begin{Example} \label{first ex} \rm{
Consider the cyclic contraction $\psi: A \to A'$ given in Figure \ref{example}.
The arrow $\delta$, contracted under $\psi$, is nonrigid since it belongs to the perfect matching $D = \{c, \delta\}$, and $D$ is equivalent to the perfect matching $D' = \{a,b\}$ not containing $\delta$.
It is straightforward to verify that all the other arrows of $A$ belong to rigid perfect matchings.
}\end{Example}

The following example demonstrates why it is necessary to define $\psi$ by a sequence of contractions where only one nonrigid arrow is contracted at a time.

\begin{Example} \label{Example4} \rm{
Consider the two contractions of the non-cancellative dimer quiver $Q$ given in Figure \ref{example4}.
In each case, the contracted quiver $Q'$ is cancellative, and the arrows in $Q$ and $Q'$ are labeled by their respective $\bar{\tau}\psi$- and $\bar{\tau}$-images.
Furthermore, the arrows $a,b \in Q_1$ are both nonrigid.
The cycle algebra is preserved in case (i),
$$S = k\left[ xz,xw,yz,yw \right] = S'.$$
In contrast, the cycle algebra is not preserved in case (ii),
$$S = k\left[x,y,xz,yz \right] \subsetneq k\left[x,y,z\right] = S'.$$
This shows that, in general, the cycle algebra will not be preserved if more than one nonrigid arrow is contracted at a time.
(In both cases, $S$ is isomorphic to the conifold coordinate ring $k\left[ s,t,u,v \right]/(st - uv)$.)
}\end{Example}

\begin{figure}
$$\begin{array}{ccccc}
& Q & & Q' & \\
(i) &
\xy
(-13,6.5)*+{\text{\scriptsize{$1$}}}="1";(0,6.5)*+{\text{\scriptsize{$2$}}}="2";(13,6.5)*+{\text{\scriptsize{$1$}}}="3";
(-13,-6.5)*+{\text{\scriptsize{$2$}}}="4";(0,-6.5)*+{\text{\scriptsize{$1$}}}="5";(13,-6.5)*+{\text{\scriptsize{$2$}}}="6";
(0,0)*{\cdot}="7";
{\ar^{x}"1";"2"};{\ar_{y}"3";"2"};{\ar^{z}"4";"1"};{\ar_{1}^a@[green]"2";"7"};{\ar@{..>}"2";"7"};{\ar_w^b"7";"5"};{\ar_{z}"6";"3"};{\ar^{y}"5";"4"};{\ar_x"5";"6"};
\endxy
& \stackrel{\psi}{\longrightarrow} &
\xy
(-13,6.5)*+{\text{\scriptsize{$1$}}}="1";(0,6.5)*+{\text{\scriptsize{$2$}}}="2";(13,6.5)*+{\text{\scriptsize{$1$}}}="3";
(-13,-6.5)*+{\text{\scriptsize{$2$}}}="4";(0,-6.5)*+{\text{\scriptsize{$1$}}}="5";(13,-6.5)*+{\text{\scriptsize{$2$}}}="6";
{\ar^{x}"1";"2"};{\ar_{y}"3";"2"};{\ar^{z}"4";"1"};{\ar_{w}^b"2";"5"};{\ar_{z}"6";"3"};{\ar^{y}"5";"4"};{\ar_x"5";"6"};
\endxy
& Q_1^* = \left\{ a \right\}
\\
\\
(ii) &
\xy
(-13,6.5)*+{\text{\scriptsize{$1$}}}="1";(0,6.5)*+{\text{\scriptsize{$2$}}}="2";(13,6.5)*+{\text{\scriptsize{$1$}}}="3";
(-13,-6.5)*+{\text{\scriptsize{$2$}}}="4";(0,-6.5)*+{\text{\scriptsize{$1$}}}="5";(13,-6.5)*+{\text{\scriptsize{$2$}}}="6";
(0,0)*{\cdot}="7";
{\ar^{x}"1";"2"};{\ar_{y}"3";"2"};{\ar^{z}"4";"1"};{\ar_{1}^a@[green]"2";"7"};{\ar@{..>}"2";"7"};{\ar_1^b@[green]"7";"5"};{\ar@{..>}"7";"5"};{\ar_{z}"6";"3"};{\ar^{y}"5";"4"};{\ar_x"5";"6"};
\endxy
& \stackrel{\psi}{\longrightarrow} &
\xy
(-6.5,6.5)*+{\text{\scriptsize{$1$}}}="1";(6.5,6.5)*+{\text{\scriptsize{$1$}}}="2";(-6.5,-6.5)*+{\text{\scriptsize{$1$}}}="3";(6.5,-6.5)*+{\text{\scriptsize{$1$}}}="4";
{\ar^z"3";"1"};{\ar_y"2";"1"};{\ar^x"1";"4"};{\ar_z"4";"2"};{\ar^y"4";"3"};
\endxy
& Q_1^* = \left\{ a,b \right\}
\end{array}$$
\caption{(Example \ref{Example4}.)
Two contractions of the non-cancellative dimer quiver $Q$.
Each quiver is drawn on a torus.
In case (i) $\psi$ is cyclic, and in case (ii) $\psi$ is not cyclic since the cycle algebra is not preserved.}
\label{example4}
\end{figure}

\begin{Example} \label{Example2} \rm{
Consider the cyclic contraction $\psi: A \to A'$ defined by the maximal sequence of contractions given in Figure \ref{example2}.
$Q'$ is a cancellative dimer quiver with a length $1$ unit cycle.
Observe that both loops, drawn in blue, are redundant generators for the dimer algebra $A' = kQ'/I'$; however, there is no (well-defined) contraction from $A$ to $A'$ with the loops removed from $Q'$.
}\end{Example}

\begin{Proposition} \label{rigid}
If a perfect matching is rigid, then it is simple.
\end{Proposition}

\begin{proof}
Let $A$ be a dimer algebra, and let $D$ be a perfect matching of $Q$ which is not simple.
We want to show that $D$ is not rigid.

Let $V$ be an $A$-module of dimension $1^{Q_0}$ with support $Q \setminus D$.
Fix a simple submodule $S$ of $V$.
Denote by $Q^S \subset Q$ the supporting subquiver of $S$, and by $Q \setminus Q^S$ the subquiver of $Q$ obtained by removing the arrows in $Q^S$.

Let $\alpha$ be the set of arrows in $Q \setminus Q^S$ whose tails lie in $Q^S$, and let $\beta$ be the set of arrows in $Q \setminus Q^S$ whose heads lie in $Q^S$.
(The sets $\alpha$ and $\beta$ need not be disjoint.)

(i) We claim that $\alpha \subseteq D$.
Indeed, let $a \in \alpha$.
Then $\operatorname{t}(a) \in Q_0^S$ and $a \not \in Q_1^S$.
Thus, since $S$ is a simple submodule of $V$, we have $aV = 0$.
Whence $a \in D$, proving our claim.

Now consider the set of arrows
\begin{equation} \label{D'}
D' := \left( D \setminus \alpha \right) \cup \beta \subset Q_1.
\end{equation}

(ii) We claim that $D'$ is a perfect matching of $Q$.
Let $[\sigma]$ be a unit cycle subquiver of $Q$.
It suffices to show that $[\sigma]$ contains precisely one arrow in $D'$.

First suppose $[\sigma]$ does not intersect $Q^S$.
Then by (\ref{D'}), the unique arrow in $[\sigma]$ which belongs to $D$ is the unique arrow in $[\sigma]$ which belongs to $D'$.

So suppose $[\sigma]$ intersects $Q^S$ in a (possibly trivial) path; let $p$ be such a path of maximal length.
Then the head of $p$ is the tail of an arrow $a$ in $[\sigma]$ which belongs to $\alpha$.
Whence $a$ belongs to $D$ by Claim (i).
Thus $p$ is unique since $D$ is a perfect matching.

Let $b$ be the arrow in $[\sigma]$ whose head is the tail of $p$.
Then $b$ is in $\beta$.
Thus $b$ belongs to $D'$ by (\ref{D'}).
Furthermore, $b$ is the unique arrow in $[\sigma]$ which belongs to $D'$ since $p$ is unique.

Therefore in either case, $[\sigma]$ contains precisely one arrow in $D'$.

(iii) We claim that $D$ and $D'$ are equivalent perfect matchings.
Let $p$ be a cycle in $A$.
If $p$ is contained in $Q^S$, then
$$n_D(p) = 0 = n_{D'}(p).$$

So suppose $p$ is a cycle in $Q$ that is not wholly contained in $Q^S$.
Then $p$ must contain an arrow in $\beta$ for each instance it enters the subquiver $Q^S$, and must contain an arrow in $\alpha$ for each instance it exits $Q^S$.
Since $p$ is a cycle, the number of times $p$ enters $Q^S$ equals the number of times $p$ exits $Q^S$.
It follows that
$$n_D(p) = n_{D'}(p).$$
Therefore $D$ and $D'$ are equivalent.

(iv) Finally, we claim that $D$ is  not rigid.
By Claim (iii) it suffices to show that $D' \not = D$.
Since $D$ is not simple, we have $S \not = V$.
Whence $\alpha \not = \beta$.
Therefore $D' \not = D$.
\end{proof}

\begin{Proposition} \label{rigid2}
If a perfect matching is simple, then it is rigid.
Consequently, a perfect matching is simple if and only if it is rigid.
\end{Proposition}

\begin{proof}
Let $A$ be a dimer algebra, and let $D$ be a perfect matching of $Q$ which is not rigid.
We want to show that $D$ is not simple.

Assume to the contrary that $D$ is simple.
Then there is a cycle $p$ that contains each arrow in $Q_1 \setminus D$.
Whence, $n_D(p) = 0$.
Furthermore, since $D$ is nonrigid, $D$ is equivalent to a perfect matching $D' \not = D$.
In particular,
$$n_{D'}(p) = n_D(p).$$

Let $a \in D' \setminus D$.
Since $a \not \in D$, $a$ is a subpath of $p$.
Thus, since $a \in D'$, we have $n_{D'}(p) \geq 1$.
But then
$$0 = n_D(p) = n_{D'}(p) \geq 1,$$
a contradiction.
Therefore $D$ is not simple.

The equivalence of simple and rigid then follows from Proposition \ref{rigid}.
\end{proof}

\begin{Proposition} \label{simple matching}
If a dimer algebra is non-cancellative, then it has an arrow that is not contained in any simple matching.
\end{Proposition}

\begin{proof}
See \cite[Corollary 3.5]{B3}.
\end{proof}

\begin{Theorem} \label{theorem1}
The dimer algebra $A'$, defined by the sequence (\ref{sequence}), is cancellative.
\end{Theorem}

\begin{proof}
Recall that an arrow is nonrigid if it is not contained in any perfect matching; and rigidity is not defined for pseudo-arrows. 
Thus, since the sequence (\ref{sequence}) is maximal, each arrow of $Q'$ is contained in a rigid perfect matching (though $Q'$ may contain pseudo-arrows, and therefore length $1$ paths that do not belong to any perfect matching).
But then each arrow of $Q'$ is contained in a simple matching, by Proposition \ref{rigid}.
Therefore $A'$ is cancellative, by Proposition \ref{simple matching}.
\end{proof}

If $\psi: A \to A'$ is a contraction of dimer algebras and $A'$ has a perfect matching, then $\psi$ does not contract an unoriented cycle of $Q$ to a vertex \cite[Lemma 3.9]{B3}.
In the following, we prove the converse.

\begin{Lemma} \label{induce}
Consider the $k$-linear map of dimer path algebras $\psi: kQ \to kQ'$ defined by contracting a set of arrows in $Q$ to vertices.
If no unoriented cycle in $Q$ is contracted to a vertex, then $\psi$ induces a $k$-linear map of dimer algebras
$$\psi: A = kQ/I \to A' = kQ'/I'.$$
\end{Lemma}

\begin{proof}
Factor $\psi: kQ \to kQ'$ into a sequence of $k$-linear maps of dimer path algebras
$$kQ \stackrel{\psi_0}{\longrightarrow} kQ^1 \stackrel{\psi_1}{\longrightarrow} kQ^2 \stackrel{\psi_2}{\longrightarrow} \cdots \stackrel{\psi_m}{\longrightarrow} kQ',$$
where each $\psi_j$ contracts a single arrow of $Q^j$.
To show that $\psi$ induces a $k$-linear map $\psi: A \to A'$, that is, $\psi(I) \subseteq I'$, it suffices to show that for each $0 \leq j \leq m$, we have
$$\psi_j(I_j) \subseteq I_{j+1}.$$
We may therefore assume that $\psi: kQ \to kQ'$ contracts a single arrow $\delta$.

Let $p-q$ be a generator for $I$ given in (\ref{I}); that is, $p,q$ are paths and there is an $a \in Q_1$ such that $pa$ and $qa$ are unit cycles.
We claim that $\psi(p-q)$ is in $\psi(I)$.

If $\delta \not = a$, then $\psi(pa) = \psi(p)\psi(a)$ and $\psi(qa) = \psi(q)\psi(a)$ are unit cycles, and $\psi(a) \in Q'_1$ has length 1.
Thus
\begin{equation} \label{psi p-q}
\psi(p-q) = \psi(p) - \psi(q) \in I'.
\end{equation}

So suppose that $\delta = a$, and no cycle in $Q$ is contracted to a vertex under $\psi$.
Then $\delta$ is not a loop.
Whence, $\psi(p)$ and $\psi(q)$ are unit cycles.
But all unit cycles at a fixed vertex are equal, modulo $I'$.
Therefore (\ref{psi p-q}) holds in this case as well.
\end{proof}

\begin{Lemma} \label{Tea}
If $A$ is cancellative, then $\mathcal{C}^{u,0} = \hat{\mathcal{C}}^u$.
\end{Lemma}

\begin{proof}
We have $\mathcal{C}^{u,0} \subseteq \hat{\mathcal{C}}^u$, by Lemma \ref{yoohooo} (or alternatively, \cite[Lemma 4.8.1]{B1}).
Conversely, since $A$ is cancellative, we have $\mathcal{C}^{u,0} \supseteq \hat{\mathcal{C}}^u$ by \cite[Proposition 4.20.1]{B1}.
\end{proof}

\begin{Proposition} \label{when?}
Let $\psi: kQ \to kQ'$ be the $k$-linear map defined by the sequence (\ref{sequence}).
If $p$ is a cycle in $\mathcal{C}^{u,0}$, then $\psi(p)$ is a cycle in $\mathcal{C}'^{u,0}$.
\end{Proposition}

\begin{proof}
Let $p \in \mathcal{C}^{u,0}$.
Then there is a perfect matching $D_1 \in \mathcal{P}$ such that
\begin{equation} \label{D_1}
x_{D_1} \nmid \bar{\eta}(p).
\end{equation}

Since $\psi: kQ \to kQ'$ is defined by the sequence (\ref{sequence}), where each $\psi_j$ contracts a single nonrigid arrow of $Q^j$, there is a perfect matching $D_2 \in \mathcal{P}$ equivalent to $D_1$ such that $\psi(D_2)$ is a perfect matching of $Q'$.\footnote{The set $\psi(D_1)$, which may consist of both arrows and vertices, may not be a perfect matching of $Q'$.}
Whence
$$n_{D_2}(p) = n_{D_1}(p) \stackrel{\textsc{(i)}}{=} 0,$$
where (\textsc{i}) holds by (\ref{D_1}).
That is, $p$ does not have an arrow subpath that belongs to $D_2$.
Thus $\psi(p)$ does not have an arrow subpath that belongs to $\psi(D_2)$:
$$n_{\psi(D_2)}(\psi(p)) = 0.$$
Consequently, $\sigma_{\mathcal{P}'} \nmid \bar{\eta}' \psi(p)$.
Hence $\psi(p)$ is in $\mathcal{C}'^{v,0}$ for some $v \in \mathbb{Z}^2$.
But $v = u$ since no cycle in $Q$ contracts to a vertex under $\psi$, by Proposition \ref{cycle-nondegenerate}.1.
Therefore $\psi(p)$ is in $\mathcal{C}'^{u,0}$.
\end{proof}

\begin{Theorem} \label{theorem2}
Let $\psi: kQ \to kQ'$ be the $k$-linear map defined by the sequence (\ref{sequence}).
Then $\psi$ induces a contraction of dimer algebras $\psi: A \to A'$, and $S = S'$.
\end{Theorem}

\begin{proof}
(i) No cycle in $Q$ contracts to a vertex under $\psi$, by Proposition \ref{cycle-nondegenerate}.1.
Therefore $\psi: kQ \to kQ'$ induces a contraction of dimer algebras $\psi: A \to A'$, by Lemma \ref{induce}.

(ii) We claim that $S = S'$.
Set $\sigma := \sigma_{\mathcal{S}'}$.

The inclusion $S \subseteq S'$ holds since the $\psi$-image of a cycle in $Q$ is a cycle in $Q'$.

To show that reverse inclusion, let $g \in S'$.
By Theorem \ref{theorem1}, $A'$ is cancellative.
Thus $S'$ is generated over $k$ by $\sigma$ and a set of monomials in the polynomial ring $k[x_D \ | \ D \in \mathcal{S}']$ that are not divisible by $\sigma$, by \cite[Theorem 5.9, Proposition 5.14]{B1}.
Furthermore, $\sigma$ is in $S$ since $\sigma$ is the $\bar{\tau}\psi$-image of each unit cycle in $Q$.
Therefore it suffices to assume that $g$ is a monomial that is not divisible by $\sigma$.

Since $g$ is a monomial in $S'$, there is some $u \in \mathbb{Z}^2$ such that $g$ is the $\bar{\tau}$-image of a cycle $q$ in ${\mathcal{C}'}^u$.
Furthermore, since $\sigma \nmid g$, we have $q \in \hat{\mathcal{C}}'^u$ by Lemma \ref{Tea}. 

Now there exists a cycle $p$ in $\mathcal{C}^{u,0}$, by Corollary \ref{dude}.
Whence $\psi(p)$ is in $\mathcal{C}'^{u,0}$, by Proposition \ref{when?}.
Thus, since $A'$ is cancellative, $\psi(p)$ is in $\hat{\mathcal{C}}'^u$, by Lemma \ref{Tea}.
Furthermore, since $A'$ is cancellative, any two cycles in $\hat{\mathcal{C}}'^u$ have the same $\bar{\tau}$-image, by \cite[Proposition 4.20.2]{B1}.
Hence
$$g = \bar{\tau}(q) = \bar{\tau}\psi(p) \in S.$$
Therefore $S' \subseteq S$, and so $S = S'$.
\end{proof}

Theorems \ref{theorem1} and \ref{theorem2} together imply that every cycle-nondegenerate, hence nondegenerate, dimer algebra admits a cyclic contraction.

\begin{Example} \rm{
A dimer algebra for which Theorem \ref{main} \textit{does not} apply is given in Figure \ref{example3}.
Its quiver contains no perfect matchings, and is thus degenerate.
}\end{Example}

\begin{figure}
$$\xy 0;/r.365pc/:
(-6,-6)*+{\text{\scriptsize{$1$}}}="1";(6,-6)*+{\text{\scriptsize{$1$}}}="2";(6,6)*+{\text{\scriptsize{$1$}}}="3";(-6,6)*+{\text{\scriptsize{$1$}}}="4";
{\ar"1";"2"};{\ar"2";"3"};{\ar"4";"3"};{\ar"1";"4"};
{\ar@/^.9pc/"4";"1"};{\ar@/^.9pc/"3";"4"};
\endxy$$
\caption{The quiver of a degenerate dimer algebra, drawn on a torus.}
\label{example3}
\end{figure}

The following corollary allows us to refer to \textit{the} cycle of algebra of a dimer algebra.

\begin{Corollary} \label{last}
Every cycle-nondegenerate, hence nondegenerate, dimer algebra $A$ has a cycle algebra $S$, and $S$ is independent of the choice of cyclic contraction, up to isomorphism.
\end{Corollary}

\begin{proof}
Let $A$ be cycle-nondegenerate.
We have shown that $A$ admits a cyclic contraction $\psi: A \to A'$, and thus $A$ has a cycle algebra $S$.
Recall the isomorphism in (\ref{S cong}),
$$S \cong k[ \overbar{\mathbb{S}(A)} ]^{\operatorname{GL}}.$$
Since the right-hand side is independent of $A'$, $S$ does not depend on $\psi$.
\end{proof}

It was shown in \cite[Theorem 1.3]{BIU} that if $Q$ is a nondegenerate dimer quiver, then a set of its arrows may be contracted to produce a cancellative dimer quiver $Q'$ with the same characteristic polygon.
In future work, we hope to determine how this theorem is related to Theorem \ref{main}.

\ \\
\textbf{Acknowledgments.}  The author would like to thank Akira Ishii, Kazushi Ueda, and Ana Garcia Elsener for useful discussions, as well as an anonymous referee for comments that have helped improve the article.
The author was supported by the Austrian Science Fund (FWF) grant P 30549-N26.

\bibliographystyle{hep}

\begin{thebibliography}{10}
\bibitem[BKM]{BKM} K.\ Baur, A.\ King, B.\ Marsh, Dimer models and cluster categories of Grassmannians, \textit{Proc.\ London Math.\ Soc.}, 113(2):213-260, 2016.
\bibitem[B1]{B1} C.\ Beil, Dimer algebras, ghor algebras, and cyclic contractions, arXiv:1711.09771.
\bibitem[B2]{B2} \bysame, Morita equivalences and Azumaya loci from Higgsing dimer algebras, \textit{J.\ Algebra}, 453:429-455, 2016.
\bibitem[B3]{B3} \bysame, Noetherian criteria for dimer algebras, \textit{J.\ Algebra}, 585:294-315, 2021.
\bibitem[B4]{B4} \bysame, Nonnoetherian homotopy dimer algebras and noncommutative crepant resolutions, \textit{Glasgow Math.\ J.}, 60(2):447-479, 2018.
\bibitem[B5]{B5} \bysame, On the central geometry of nonnoetherian dimer algebras, \textit{J.\ Pure Appl.\ Algebra}, 225(8), 2021.
\bibitem[B6]{B6} \bysame, On the noncommutative geometry of square superpotential algebras, \textit{J.\ Algebra}, 371:207-249, 2012.
\bibitem[B7]{B7} \bysame, The central nilradical of nonnoetherian dimer algebras, arXiv:1902.11299.
\bibitem[BIU]{BIU} C.\ Beil, A.\ Ishii, K.\ Ueda, Cancellativization of dimer models, arXiv:1301.5410.
\bibitem[Bo]{Bo} R.\ Bocklandt, A dimer ABC, \textit{Bull.\ London Math.\ Soc.}, 48(3):387-451, 2016.
\bibitem[Bo2]{Bo2} \bysame, Graded Calabi Yau algebras of dimension 3, \textit{J.\ Pure Appl.\ Algebra}, 212(1):14-32, 2008.
\bibitem[BGH]{BGH} S.\ Bose, J.\ Gundry, Y.\ He, Gauge Theories and Dessins d'Enfants: Beyond the Torus, \textit{J.\ High Energy Phys.}, 01:135, 2015.
\bibitem[Br]{Br} N.\ Broomhead, Dimer models and Calabi-Yau algebras, \textit{Memoirs AMS}, 1011, 2012.
\bibitem[CBQ]{CBQ} A.\ Craw, R.\ Bocklandt, A.\ Quintero V\'elez, Geometric Reid's recipe for dimer models, \textit{Math.\ Ann.}, 361:689-723, 2015.
\bibitem[FHKV]{FHKV} B.\ Feng, Y.\ He, K.\ D.\ Kennaway, C.\ Vafa, Dimer models from mirror symmetry and quivering amoebae, \textit{Adv.\ Theor.\ Math.\ Phys.}, 12(3):489-545, 2008.
\bibitem[FHMSVW]{FHMSVW} S.\ Franco, A.\ Hanany, D.\ Martelli, J.\ Sparks, D.\ Vegh, B.\ Wecht, Gauge theories
from toric geometry and brane tilings, \textit{J.\ High Energy Phys.}, 01:128, 2006.
\bibitem[FHVWK]{FHVWK} S.\ Franco, A.\ Hanany, D.\ Vegh, B.\ Wecht, K.\ Kennaway, Brane dimers and quiver
gauge theories, \textit{J.\ High Energy Phys.}, 01:096, 2006.
\bibitem[FU]{FU} M.\ Futaki, K.\ Ueda, Exact Lefschetz fibrations associated with dimer models, \textit{Math.\ Res.\ Lett.}, 17(6):1029-1040, 2010.
\bibitem[GK]{GK} A.\ Goncharov, R.\ Kenyon, Dimers and cluster integrable systems, \textit{Annales scientifiques de l'ENS}, 46(5):747-813, 2013.
\bibitem[HK]{HK} A.\ Hanany, K.\ D.\ Kennaway, Dimer models and toric diagrams, arXiv:0503149.
\bibitem[IU]{IU} A.\ Ishii, K.\ Ueda, Dimer models and the special McKay correspondence, \textit{Geometry and Topology}, 19(6):3405-3466, 2015.
\bibitem[IN]{IN} O.\ Iyama, Y.\ Nakajima, On steady non-commutative crepant resolutions, \textit{J.\ Noncommutative Geom.}, 12(2):457-471, 2018.
\end{thebibliography}
\def\cprime{$'$} \def\cprime{$'$}

\end{document}